\documentclass[a4paper]{article}

\usepackage[textwidth=125mm,textheight=195mm]{geometry}
\usepackage{authblk}
\usepackage[latin1]{inputenc}  
\usepackage[T1]{fontenc}       

\usepackage{amsmath}
\usepackage{amsthm}
\usepackage{amssymb}
\usepackage{amsfonts}
\usepackage{ascii}

\usepackage{dsfont}

\newtheorem{theorem}{Theorem}[section]
\newtheorem{lemma}[theorem]{Lemma}
\newtheorem{proposition}[theorem]{Proposition}

\newtheorem{remark}[theorem]{Remark}
\newtheorem{assumption}[theorem]{Assumption}
\newtheorem{example}[theorem]{Example}
\newtheorem{definition}[theorem]{Definition}

\newcommand{\clo}{\mathrm{clo}}        
\newcommand{\Span}{\mathrm{span}}       
\renewcommand{\Im}{{\ensuremath{\mathrm{Im\,}}}} 
\renewcommand{\Re}{{\ensuremath{\mathrm{Re\,}}}} 

\providecommand{\norm}[1]{\lVert#1\rVert} 
\providecommand{\abs}[1]{\lvert#1\rvert} 

\DeclareMathOperator{\Ran}{Ran}

\DeclareMathOperator{\sign}{sign}

\DeclareMathOperator{\Ker}{Ker}
\DeclareMathOperator{\Dom}{Dom}
\DeclareMathOperator{\dom}{dom}

\newcommand{\au}{\underline{a}}

\newcommand{\R}{\mathbb{R}}

\newcommand{\C}{\mathbb{C}}

\newcommand{\N}{\mathbb{N}}

\newcommand{\Me}{{\mathcal M}}
\newcommand{\He}{{\mathcal H}}
\newcommand{\Ke}{{\mathcal K}}
\newcommand{\Ge}{{\mathcal G}}
\newcommand{\Ie}{{\mathcal I}}
\newcommand{\Ee}{{\mathcal E}}
\newcommand{\De}{{\mathcal D}}

\newcommand{\We}{{\mathcal W}}

\title{Maximal--quasi--accretive Laplacians on finite metric graphs}
\begin{document}
\author{Amru Hussein \and Institut f\"{u}r Mathematik, \\ Johannes Gutenberg-Universit\"{a}t Mainz, Germany}
  \maketitle
  
 \begin{abstract}
For a finite not necessarily compact metric graph, one considers the differential expression $-\frac{d^2}{d x^2}$ on each edge. The boundary conditions at the vertices of the graph yielding quasi--m--accretive as well as m--accretive operators are completely characterized. 
 \end{abstract}

\section{Introduction} 
The subject of Laplacians on metric graphs has attracted a lot of attention in the last decades. Without going into details the author refers to  \cite{QG} and the articles \cite{PKQG1,PKWaves} and the references given therein, where also a brief overview on the history of the different branches of the development and their applications can be found. The present paper is devoted to the characterization of quasi--m--accretive Laplacians on finite metric graphs.

Recall that an operator $T$ in a Hilbert space $\He$ with scalar product $\langle \cdot, \cdot \rangle$ is called \textit{quasi--accretive} if there exists a real constant $C$ such that
\begin{eqnarray*}
\Re \langle u, T u\rangle + C \langle u,  u\rangle \geq 0, & \mbox{for all }  u \in \Dom(T),
\end{eqnarray*}
and $T$ is called \textit{accretive} if $C$ can be chosen to be zero. The operator $T$ is called \textit{(quasi--)m--accretive} if it is closed and has no proper (quasi--)accretive extension. In this sense (quasi--)m--accretive operators are maximal (quasi--)accretive. A closed operator $T$ is m--accretive if and only if it obeys the estimate
\begin{eqnarray*}
\norm{\left(T-\lambda \right)^{-1}} \leq \left(\Re \lambda \right)^{-1} & \mbox{for } \Re \lambda>0,
\end{eqnarray*}
compare \cite[Chapter V, \S 3.10]{Kato}. An operator $L$ is called \textit{(quasi--)m--dissipative} if the operator $T=-L$ is (quasi--)m--accretive. Let be $C\in\R$. Then the operator $T$ is called \textit{sectorial with vertex $C$}, if
\begin{eqnarray*}
\abs{\Im \langle u, T u\rangle} \leq \Re \langle u, T u\rangle + C \langle u,  u\rangle , & \mbox{for all }  u \in \Dom(T).
\end{eqnarray*}
An operator $T$ is called \textit{sectorial} if it is sectorial with some vertex $C$. Furthermore a sectorial operator $T$ is called \textit{m--sectorial} if it is also quasi--m--accretive. These definitions are compatible with those given in \cite[Chapter V, \S 3.10]{Kato}. Recall that these notions refer to the numerical range of a closed operator $T$ rather than to its spectrum. In general the numerical range can be larger than the convex hull of the spectrum. 

It is known that a semigroup $S(t)$ is strongly continuous and quasi--contractive if and only if its generator $L$ is quasi--m--dissipative and then $S(t)=e^{L t}$, see for example \cite[Chapter II Corollary 3.6]{EngelNagel}. Recall that a strongly continuous semigroup $S(t)$ is called \textit{quasi--contractive} if $\norm{S(t)}\leq e^{\omega t}$ for appropriate $\omega\in \R$, compare for example \cite[Chapter II Corollary 3.6]{EngelNagel}. If $\norm{S(t)}\leq 1$, for $t\geq 0$, then $S(t)$ is called \textit{contractive}.

The scope of this work is the heat conduction equation on finite metric graphs with initial conditions. Let $-\widetilde{\Delta}$ be a Laplace operator, which acts in the space of square integrable functions on a finite metric graphs. Then one considers
\begin{eqnarray}\label{Diffusioneq}
\left( \frac{\partial}{\partial t} - \widetilde{\Delta} \right) \psi(x,t) =0, & \psi(\cdot,0) = \psi_0, & \mbox{for } t\geq0.
\end{eqnarray}
The quasi--m--accretive Laplace operators $-\widetilde{\Delta}$ give exactly the quasi--m--dissipative generators $\widetilde{\Delta}$ of strongly continuous and quasi--contractive semigroups. Hence for $-\widetilde{\Delta}$ quasi--m--accretive the solution of \eqref{Diffusioneq} in the $L^2$-space is given in terms of semigroups as
\begin{eqnarray*}
\psi(\cdot,t)= e^{\widetilde{\Delta}t}\psi_0, & \mbox{where} & \norm{e^{\widetilde{\Delta}t}} \leq  e^{\omega t},
\end{eqnarray*}
for appropriate growth bound $\omega\in \R$. In particular this has applications to stochastic processes on networks. For further information on this subject especially on Brownian motions on metric graphs, see \cite{VK2007, KPS2012}.

The present work can be understood as an extension of the results obtained by V.~Kostrykin and R.~Schrader in \cite{VK2007}, where a sufficient criterion for m--accretive boundary conditions has been derived. In particular in \cite{VK2007} it has been stated that all m--accretive Laplacians can be parametrized in terms of boundary conditions. The main result of this work is the characterization of all quasi--m--accretive Laplacians on finite metric graphs in terms of boundary conditions. The proof shows that these operators are even m--sectorial, and the real parts of such m--sectorial Laplacians are self-adjoint Laplace operators. Combining the main result with results from \cite{AL2009}, one obtains a complete characterization of all m--accretive boundary conditions, too.

The subject of Laplacians on metric graphs lies - from the mathematical point of view - in the intersection of different branches of mathematics. Here it is worth mentioning spectral theory and the theory of ordinary differential equations or systems of them. One approach is to put the question of appropriate boundary conditions into the framework of extension theory. 

In the most general context of extension theory many results characterizing m--accretive extensions of non--negative closed symmetric operators and m--sectorial extensions with vertex zero have been obtained by  Y.~Arlinskii, Y.~Kovalev and \`E.~R.~Tsekanovski{\u\i}, compare the recent work \cite{Tsekanovski2010} and the references therein. In the particular context of boundary triples  V.~A.~Derkach, M.~M.~Malamud  and \`E.~R.~Tsekanovski{\u\i}, see \cite{Tsekanovski1989} and M.~M.~Malamud, see \cite{Malamud1992} obtained characterizations earlier. Some of the results proven here have been obtained already in the above mentioned works. Some are more specific due to the special structure of the problem. In general the techniques used here are very explicit, because the situation discussed is very explicit. Thus the proofs are more transparent than those given in the above mentioned works, which have been obtained in a long term process. It is not completely clear how the general theory can be applied effectively to the problem stated here. 

The content of this note is part of the author's PhD thesis, see \cite[Chapter 1]{Ich}. 

The work is organized as follows: in the subsequent section the different types of boundary conditions are discussed and the starting point of the study is described. This is followed by the formulation of the main results and the discussion of examples. The proofs are given separately in the last section.

\subsection*{Acknowledgements}
I would like to thanks Vadim Kostrykin for providing the interesting and widely ramified question, helpful discussion and pointing out the works \cite{AL2009} and \cite{Tsekanovski1992}. I am also indebted to Petr Siegl for communicating and discussing Example \ref{ex2AccQG}. I thank Delio Mugnolo for his hospitality during my visit to the research group ``Symmetriemethoden in Quantengraphen'' in Ulm in January 2012, where I presented parts of this work, as well as for helpful remarks and discussions on this subject.

\section{Basic structures} 

The notation is as in \cite[Section 2]{VK2007} and it is summarized here briefly. A graph is a $4$-tuple $\Ge = \left( V, \Ie,\Ee, \partial \right)$, where $V$ denotes the set of \textit{vertices}, $\Ie$ the set of \textit{internal edges} and $\Ee$ the set of \textit{external edges}, where the set $\Ee \cup \Ie$ is summed up in the notion \textit{edges}. The \textit{boundary map} $\partial$ assigns to each internal edge $i\in \Ie$ an ordered pair of vertices $\partial (i)= (v_1,v_2)\in V\times V$, where $v_1$ is called its \textit{initial vertex} and $v_2$ its \textit{terminal vertex}. Each external edge $e\in \Ee$ is mapped by $\partial$ onto a single, its initial, vertex. A graph is called finite if $\abs{V}+\abs{\Ie}+\abs{\Ee}<\infty$ and a finite graph is \textit{compact} if $\Ee=\emptyset$.

The graph is endowed with the following metric structure. Each internal edge $i\in \Ie$ is associated with an interval $[0,a_i]$ with $a_i>0$, such that its initial vertex corresponds to $0$ and its terminal vertex to $a_i$. Each external edge $e\in \Ee$ is associated to the half line $[0,\infty)$, such that $\partial(e)$ corresponds to $0$. The numbers $a_i$ are called \textit{lengths} of the internal edges $i\in \Ie$ and they are summed up into the vector $\au=\{a_i\}_{i\in \Ie}\in \R_+^{\abs{\Ie}}$. The 2-tuple consisting of a finite graph endowed with a metric structure is called a \textit{metric graph} $(\Ge,\au)$. The metric on $(\Ge,\au)$ is defined via minimal path lengths. 

Given a finite metric graph $(\Ge,\au)$ one considers the Hilbert space
\begin{eqnarray}
 \He \equiv \He(\Ee,\Ie,\au)= \He_{\Ee} \oplus \He_{\Ie}, & \displaystyle{\He_{\Ee}= \bigoplus_{e\in\Ee} \He_e,} & \He_{\Ie}= \bigoplus_{i\in\Ie} \He_i,
\end{eqnarray}     
where $\He_j= L^2(I_j)$ with 
$$I_j= \begin{cases} [0,a_j], & \mbox{if} \ j\in \Ie, \\ [0,\infty), &\mbox{if} \ j\in \Ee.  \end{cases}$$

By $\De_j$ with $j\in \Ee \cup \Ie$ denote the set of all $\psi_j\in \He_j$ such that $\psi_j$ and its derivative $\psi_j^{\prime}$ are absolutely continuous and $\psi_j^{\prime\prime}$ is square integrable. Let $\De_j^0$ denote the set of all elements $\psi_j\in \De_j$ with
\begin{eqnarray*}
\psi_j(0)=0, &  \psi^{\prime}(0)=0, & \mbox{for}  \ j\in \Ee,    \\
\psi_j(0)=0, &  \psi^{\prime}(0)=0, & \psi_j(a_j)=0, \  \psi^{\prime}(a_j)=0, \ \mbox{for} \ j\in \Ie. 
\end{eqnarray*}
Let $\Delta$ be the differential operator 
\begin{eqnarray*}
\left( \Delta \psi\right)_j (x) = \frac{d^2}{dx}\psi_j(x), & j\in \Ee\cup \Ie, & x\in I_j
\end{eqnarray*}
with domain $\De$ and $\Delta^0$ its restriction on the domain $\De^0$, where 
\begin{eqnarray*}
\De= \bigoplus_{j\in \Ee \cup \Ie} \De_j, & \mbox{and} &  \De^0= \bigoplus_{j\in \Ee \cup \Ie} \De_j^0.
\end{eqnarray*}

One can check easily that the operator $\Delta^0$ is a closed symmetric operator with deficiency indices $(d,d)$, where $d=\abs{\Ee}+2\abs{\Ie}$, and its Hilbert space adjoint is $(\Delta^0)^*=\Delta$. The aim of this note is to discuss extensions $-\widetilde{\Delta}$ of the Laplacian $-\Delta^0$ with
$$-\Delta^0\subset-\widetilde{\Delta}\subset -\Delta,$$
that is lying between the minimal and the maximal operator. In the context of extension theory extensions with this property are called quasi--self--adjoint, see for example \cite{Tsekanovski2010}. In the situation considered here these extensions can be discussed in terms of boundary conditions. For this purpose one defines the auxiliary Hilbert space
\begin{equation*}
\Ke \equiv \Ke(\Ee, \Ie) = \Ke_{\Ee}  \oplus \Ke_{\Ie}^- \oplus \Ke_{\Ie}^+
\end{equation*}
with $\Ke_{\Ee} \cong \C^{\abs{\Ee}}$ and $\Ke_{\Ie}^{(\pm)} \cong \C^{\abs{\Ie}}$. For $\psi\in \De$ one defines the vectors of boundary values
\begin{eqnarray*}
\underline{\psi}= \begin{bmatrix} \{\psi_{e}(0)\}_{e\in\Ee} \\ 
\{\psi_{i}(0)\}_{i\in\Ie} \\
\{\psi_{i}(a_i)\}_{i\in\Ie}
\end{bmatrix} &\mbox{and} & 
\underline{\psi^{\prime}}= \begin{bmatrix} \{\psi_{e}^{\prime}(0)\}_{e\in\Ee} \\ 
\{\psi_{i}^{\prime}(0)\}_{i\in\Ie} \\
\{-\psi_{i}^{\prime}(a_i)\}_{i\in\Ie}
\end{bmatrix}.
\end{eqnarray*}
One sets 
\begin{equation*}
[\psi]:= \underline{\psi} \oplus \underline{\psi^{\prime}} \in \Ke \oplus \Ke
\end{equation*}
and denotes by the redoubled space $\Ke^2= \Ke \oplus \Ke$ the \textit{space of boundary values}. 

\section{Boundary conditions}
Let $A$ and $B$ be linear maps in $\Ke$. By $(A, \, B)$ one denotes the linear map from $\Ke^2=\Ke \oplus \Ke$ to $\Ke$ defined by 
$$ (A, \, B) (\chi_1 \oplus \chi_2) = A\chi_1 + B \chi_2, $$   
for $\chi_1,\chi_2\in\Ke$. Set 
\begin{equation}\label{MeAccQG}
\Me(A,B):=\Ker (A, \, B).
\end{equation}
With any subspace $\Me\subset \Ke^2$ one can associate an extension $-\Delta(\Me)$ of $-\Delta^0$, which is the restriction of $-\Delta$ to the domain
\begin{equation*}
\Dom(-\Delta(\Me))= \{ \psi \in \De \mid [\psi] \in \Me   \}.
\end{equation*}
If $\Me$ is of the form given in \eqref{MeAccQG} an equivalent description is that $\Dom (-\Delta(\Me))$ consists of all functions $\psi\in \De$ that satisfy the boundary conditions
\begin{equation*}
A \underline{\psi} + B \underline{\psi}^{\prime} = 0.
\end{equation*}
In this case one also writes equivalently $-\Delta(\Me)=-\Delta(A,B)$.

In \cite[Theorem 2.3]{VK2007} it has been shown that each m--accretive extension of $-\Delta^0$ can be represented as $-\Delta(A,B)$, for some $A,B$ satisfying the following necessary 
\begin{assumption}[{\cite[Assumption 2.1]{VK2007}}]\label{assumption2}%
Let $A$ and $B$ be maps in $\Ke$. Assume that the map $(A, \, B ) \colon \Ke^2 \rightarrow \Ke$ is surjective, that is, it has maximal rank equal to $d=\abs{\Ee} + 2\abs{\Ie}$. 
\end{assumption}
The statement of \cite[Theorem 2.3]{VK2007} admits a straight forward generalization to quasi--m--accretive Laplacians.  
\begin{proposition}\label{prop1AccQG}
Any quasi--m--accretive extension of $-\Delta^0$ can be represented by $-\Delta(A,B)$, for some $A$ and $B$ satisfying Assumption~\ref{assumption2}.
\end{proposition}
For the discussion of boundary conditions it is important to note that $A,B$ and $A^{\prime},B^{\prime}$ define the same operator $-\Delta(A,B)=-\Delta(A^{\prime},B^{\prime})$ if and only if the corresponding subspaces $\Me(A,B)$ and $\Me(A^{\prime}, B^{\prime})$ agree. This gives rise to  
\begin{definition}[{\cite[Definition 2.2]{VK2007}}]\label{CACB}
Boundary conditions defined by $A,B$ and $A^{\prime},B^{\prime}$ satisfying Assumption~\ref{assumption2} are called equivalent if the corresponding subspaces $\Me(A,B)$ and $\Me(A^{\prime}, B^{\prime})$ agree. 
\end{definition}
 
Note that boundary conditions defined by $A,B$ and $A^{\prime},B^{\prime}$ satisfying Assumption~\ref{assumption2} are equivalent if and only if there exists an invertible operator $C$ in $\Ke$ such that $A^{\prime}=CA$ and $B^{\prime}=CB$, compare \cite{VK2007}.

The question if an operator is quasi--accretive or even sectorial is closely related to the sesquilinear form defined by the operator, and here one defines the sesquilinear form $\delta_{\Me}$ by   
\begin{eqnarray*}
\delta_{\Me}[\psi,\varphi]:=\langle \psi, -\Delta(\Me) \varphi \rangle_{\He},  & \psi, \varphi\in \mbox{Dom}(\Delta(\Me)).
\end{eqnarray*}
Integration by parts gives a more practical representation for the associated \\ quadratic form 
\begin{align}\label{eq:qfAccQG}
\delta_{\Me}[\psi] :=\delta_{\Me}[\psi,\psi]= \int_{\Ge} \abs{ \psi^{\prime}}^2 + \langle \underline{\psi},\underline{\psi}^{\prime}\rangle_{\Ke},
\end{align}
where
\begin{align*}
\int_{\Ge}  \abs{ \psi^{\prime}}^2 = \sum_{i\in \Ie} \int_0^{a_i} \abs{\psi^{\prime}(x_i)}^2 dx_i + \sum_{e\in \Ee} \int_0^{\infty} \abs{\psi^{\prime}(x_e)}^2 dx_e.
\end{align*}

It turns out that in the class of boundary conditions which satisfy Assumption~\ref{assumption2}, there are two types of boundary conditions. These are related to a certain block-decomposition of the matrices $A$ and $B$, where $\Me=\Me(A,B)$. Following \cite[Section 3.1]{PKQG1} one introduces two decompositions of $\Ke$. Denote by $Q$ the orthogonal projector onto the subspace $(\Ran B)^{\perp}$, by $Q^{\perp}= \mathds{1}-Q$ the orthogonal projector onto $\Ran B$ and by $P$ the orthogonal projector onto $\Ker B$ and by $P^{\perp}= \mathds{1}-P$ the orthogonal projector onto $(\Ker B)^{\perp}$. With this one is able to write the map $(A ,\, B)$ as the block-operator matrix 
\begin{align}\label{eq: Decomposition1}
(A, \, B ) = 
  \begin{pmatrix}
     Q^{\perp} A P^{\perp} &  Q^{\perp} A P &  Q^{\perp} B P^{\perp} & 0 \\
     Q A P^{\perp} &   Q A P   & 0 & 0 \\
  \end{pmatrix}.
\end{align}
The domains of $A$ and $B$ have the orthogonal decomposition $\Ke=\Ran P \oplus \Ran P^{\perp}$ and the target set of both $A$ and $B$ is $\Ke=\Ran Q \oplus \Ran Q^{\perp}$. From this one sees that the rank condition in Assumption~\ref{assumption2} is equivalent to the fact that \\ $QA=\begin{pmatrix} QAP^{\perp} & Q A P \end{pmatrix}$ considered as a map from $\Ke$ to $\Ran Q$ is surjective. 

As remarked above the choice of the matrices $A$ and $B$ is not unique. Similar to the case of self-adjoint Laplace operators on metric graphs one can parametrize $\Me=\Me(A,B)$ at least in an ``almost unique'' way. Notice that from the definitions of $P$ and $Q$ it follows that $\dim \Ran P = \dim \Ran Q$. Therefore there exists an isomorphism 
$$U\colon \Ran Q \rightarrow \Ran P.$$ Multiplying both $A$ and $B$ from the left with the matrix 
\begin{equation*}
C=\begin{pmatrix} (Q^{\perp} B P^{\perp})^{-1} & 0 \\ 0 & U \end{pmatrix},
\end{equation*}
gives $A^{\prime}= CA$ and $B^{\prime}=CB$. These $A^{\prime}, B^{\prime}$ define equivalent boundary conditions, in the sense of Definition~\ref{CACB}, that is $\Me(A^{\prime},B^{\prime})=\Me(A,B)$, but one has achieved the normalization $B^{\prime}=P^{\perp}$. This gives the block-operator matrix 
\begin{align}\label{eq: Decomposition}
(A^{\prime}, \, B^{\prime} ) = 
  \begin{pmatrix}
     P^{\perp} A^{\prime} P^{\perp} &  P^{\perp} A^{\prime} P & P^{\perp} & 0 \\
     P A^{\prime} P^{\perp} &  P A^{\prime} P  & 0 & 0 \\
  \end{pmatrix}
\end{align}
One observes that the boundary conditions separate into 
\begin{eqnarray*}
PA^{\prime} \underline{u}=0 & \mbox{and} & P^{\perp}A^{\prime} \underline{u} + P^{\perp} \underline{u}^{\prime} =0.
\end{eqnarray*}
Similar considerations have been used in the discussion of boundary conditions that define self-adjoint Laplace operators as well as in the analysis of the corresponding quadratic forms, see \cite{PKQG1}. This motivates the notation 
\begin{equation}\label{KuchmentL}
L:= (Q^{\perp}BP^{\perp})^{-1} Q^{\perp}A P^{\perp}
\end{equation}
or equivalently $L=P^{\perp} A^{\prime} P^{\perp}$. After this preparatory work one can formulate also the sufficient assumption on boundary conditions to define quasi--m--accretive realisations of $-\Delta^0$. In addition to the necessary Assumption~\ref{assumption2} one needs

\begin{assumption}\label{assumptionA}
Let $A$ and $B$ be maps in $\Ke$. Denote by $Q$ be the orthogonal projector in $\Ke$ onto $(\Ran B)^{\perp}$ and by $P$ the orthogonal projector in $\Ke$ onto $\Ker B$. Assume that $$QAP^{\perp} = 0.$$
\end{assumption}

\begin{remark}
Assuming both Assumption~\ref{assumption2} and Assumption~\ref{assumptionA} it follows that $QAP$, as a map from $\Ran P$ to $\Ran Q$ is invertible. With this the block decomposition given in \eqref{eq: Decomposition} can be simplified to
\begin{align}\label{eq: Decomposition2}
(A^{\prime}, \, B^{\prime} ) = 
  \begin{pmatrix}
     P^{\perp} A^{\prime} P^{\perp} &  P^{\perp} A^{\prime} P & P^{\perp} & 0 \\
     0 &   P  & 0 & 0 
  \end{pmatrix}.
\end{align}
Surprisingly the block $P^{\perp} A^{\prime} P$ has no influence on the numerical range of the operator $-\Delta(A^{\prime}, B^{\prime} )$, since $P\underline{\psi}=0$ for all $\psi\in \Dom(-\Delta(A^{\prime}, B^{\prime} ))$. Hence it is even sufficient to consider 
\begin{align*}
(A^{\prime\prime}, \, B^{\prime\prime} ) = 
  \begin{pmatrix}
     L &  0 & P^{\perp} & 0 \\
     0 &   P  & 0 & 0 
  \end{pmatrix}.
\end{align*}
\end{remark}

\section{Results and examples}
The main result of this note is 
 \begin{theorem}\label{th1AccQG}
\ \ \ \  \   \  \ \  \ \ \ \ \
\begin{enumerate}
\item The operator $-\Delta(A,B)$ is quasi--m--accretive if and only if $A,B$ satisfy both Assumption~\ref{assumption2} and Assumption~\ref{assumptionA}. 
\item All quasi--m--accretive realizations $-\Delta(A,B)$ are even m--sectorial. 
\end{enumerate}
\end{theorem}
Note that for any positive closed symmetric operator there exists an extension that is m--accretive, but not sectorial with vertex zero, see \cite[Theorem 1]{Tsekanovski1992}. In the particular situation considered here Theorem~\ref{th1AccQG} exhibits that there is at least some vertex for which the  m--accretive operator $-\Delta(A,B)$ is sectorial.  

Claim (2) of Theorem~\ref{th1AccQG} is of importance because any m--sectorial operator \\ $-\Delta(A,B)$ has a well--defined real part $\Re(-\Delta(A,B))$. This is the unique self--adjoint operator that is associated to the real part of the closure of the form defined by the m-sectorial operator $-\Delta(A,B)$, compare \cite[Chapter VI, \S 3.1]{Kato}. It turns out that $\Re(-\Delta(A,B))$ is itself a Laplace operator. Together with results from \cite{AL2009} Theorem~\ref{th1AccQG} yields also a characterization of all m--accretive boundary conditions.

\begin{theorem}\label{th2AccQG}
\begin{enumerate}
\item Let $A,B$ satisfy Assumption~\ref{assumption2} and Assumption~\ref{assumptionA}. Then 
\begin{eqnarray*}
\Re (-\Delta(A,B))= -\Delta(A^{\prime},B^{\prime}),  
\end{eqnarray*}
is self--adjoint with
\begin{eqnarray*}
A^{\prime}= P + \Re L  & \mbox{and} & B^{\prime}=P^{\perp}, 
\end{eqnarray*}
where $P$ denotes the orthogonal projector onto $\Ker B$, $P^{\perp}= \mathds{1}-P$ and $L$ is given by equation~\eqref{KuchmentL}. 
\item The operator $-\Delta(A,B)$ is m--accretive if and only if the following two conditions are satisfied 
\begin{enumerate} 
\item Assumption~\ref{assumption2} is fulfilled and
\item $\Re (AB^*) + B M_0(\au) B^*  \leq 0$ with  $M_0(\au)= \begin{pmatrix} 0 & 0 & 0 \\ 0 & -\frac{1}{\au} & \frac{1}{\au}  \\ 0 & \frac{1}{\au} & -\frac{1}{\au} \end{pmatrix},$  \\ where $\frac{1}{\au}$ is the $\abs{\Ie}\times \abs{\Ie}$--matrix with entries $\left\{\frac{1}{\au}\right\}_{ij}= \delta_{ij} a_i^{-1}$, $i,j\in \Ie$.
\end{enumerate}
\end{enumerate}
\end{theorem}
This theorem improves the result obtained in \cite[Theorem 2.4]{VK2007}, where it has been stated that boundary conditions satisfying Assumption~\ref{assumption2} define an m--accretive Laplace operator $-\Delta(A,B)$ whenever $\Re (AB^*)\leq 0$. This follows now already from the inequality $\Re (AB^*) + B M_0(\au) B^*  \leq \Re (AB^*),$ which uses $B M_0(\au) B^*  \leq 0$. 

Note that condition (b) for the statement (2) in Theorem~\ref{th2AccQG} assures that Assumption~\ref{assumptionA}, which is needed to apply Theorem~\ref{th1AccQG} is satisfied. When quasi--m--accretive Laplacians are important for the generation of strongly continuous quasi--contractive semigroups, m--accretive Laplacians are relevant for the generation of strongly continuous contractive semigroups. The statement of Theorem~\ref{th2AccQG} follows also from the more general results obtained in \cite[Theorem 2]{Tsekanovski1989}.

\begin{remark}\label{growthbound}
Quasi--m--accretive operators $-\Delta(A,B)$ generate strongly continuous quasi--contractive semigroups 
\begin{eqnarray*}
S(t)= e^{ t \Delta(A,B)} & \mbox{with growth estimate} & \norm{S(t)}\leq  e^{-t\omega},
\end{eqnarray*}
where the growth bound $\omega$ can be chosen as 
\begin{eqnarray*}
\omega= \min \Re \{\langle \psi, -\Delta(A,B)\psi \rangle \mid \psi \in \Dom (-\Delta(A,B))\}.
\end{eqnarray*}
This follows from \cite[Chapter II Corollary 3.6]{EngelNagel} and \cite[Chapter V, \S 10]{Kato}. Therefore the growth bound $\omega$ can be computed, according to Theorem~\ref{th2AccQG}, as the bottom of the spectrum of the self--adjoint operator $ -\Delta(A^{\prime},B^{\prime})$,
\begin{eqnarray*}
\omega= \min \sigma( -\Delta(A^{\prime},B^{\prime})).
\end{eqnarray*}
Estimates for the lowest eigenvalue of self-adjoint Laplacians on finite metric graphs are discussed in \cite[Theorem 3.10]{VKRS2006}, \cite[Corollary 10]{PKQG1} and \cite[Chapter 2]{Ich}.
\end{remark}

Examples of quasi--m--accretive operators are obtained by imposing boundary conditions that satisfy both Assumption~\ref{assumption2} and Assumption~\ref{assumptionA}. Note that this is the case for all boundary conditions with $\Ker B = 0$, as then $L=B^{-1}A$ and $P\equiv 0$. Separated boundary conditions on intervals satisfy these assumptions as well as the conditions given in Example~\ref{ex3AccQG}. The next example gives quasi--m--accretive boundary conditions with $P\neq 0$.

\begin{example}[Complex $\delta$-interaction]\label{DeltaAccQG}
Assume that the boundary conditions are local and for $\deg(\nu)\geq 2$, up to equivalence the boundary conditions at vertex $\nu$ are defined by 
\begin{align*}
A_{\nu}= \left[
   \begin{array}{cccccc}
     1 & -1 & 0 &\cdots & 0 & 0 \\
     0 & 1 & -1 &\cdots & 0 & 0  \\
     0 & 0 & 1 &\cdots & 0 & 0  \\
     \vdots &\vdots  & \vdots & & \vdots  & \vdots \\
        0 & 0 & 0 &\cdots & 1 & -1  \\
   0 & 0 & 0 &\cdots & 0 & \gamma_{\nu} 
   \end{array}
\right], && B_{\nu}= \left[
   \begin{array}{cccccc}
     0 & 0 & 0 &\cdots & 0 & 0 \\
     0 & 0 & 0 &\cdots & 0 & 0  \\
     0 & 0 & 0 &\cdots & 0 & 0  \\
     \vdots &\vdots  & \vdots & & \vdots  & \vdots \\
        0 & 0 & 0 &\cdots & 0 & 0  \\
   1 & 1 & 1 &\cdots & 1 & 1 
   \end{array}
\right],
\end{align*}  
where $\gamma_{\nu}\in \C$. For real $\gamma_{\nu}$ the Assumptions~\ref{assumption2} and \ref{assumptionA} are satisfied, that is one can represent the boundary conditions equivalently by $A^{\prime}_{\nu}= L_{\nu} + P_{\nu}$ and $B^{\prime}_{\nu}= P_{\nu}^{\perp}$, where $P^{\perp}_{\nu}$ is a rank one projector and $L_{\nu}= \frac{-\gamma_{\nu}}{\deg(\nu)} P^{\perp}_{\nu}$, compare \cite[Section 3.2.1]{PKQG1}. This carries over to the case of complex coupling parameters $\gamma_{\nu}$ and the Assumptions~\ref{assumption2} and \ref{assumptionA} are satisfied. For a connected star graph with complex $\delta$-interaction at the vertex with coupling constant $\gamma_{\nu}$ the operator $-\Delta(A_{\nu},B_{\nu})$ is associated with the quadratic form defined by 
\begin{eqnarray*}
\int_{\Ge} \abs{\psi^{\prime}}^2   - \frac {\gamma_{\nu}}{\deg(\nu)} \abs{\underline{\psi}}^2, & \mbox{where } \psi\in \{\psi\in \We \mid P\underline{\psi}=0\},
\end{eqnarray*}
and hence $\Re\left(-\Delta(A_{\nu},B_{\nu})\right)$ is the self--adjoint Laplace operator with real \\ $\delta$--interaction and coupling parameter $\Re \gamma_{\nu}$ at the vertex.
\end{example}

\begin{example}[Complex $\delta^{\prime}$-interaction]\label{ex3AccQG}
Complex $\delta^{\prime}$-potentials are locally given for $\gamma_{\nu}\in \C\setminus\{0\}$ by $\hat{A}_{\nu}= B_{\nu}$ and $\hat{B}_{\nu}= A_{\nu}$, where $A_{\nu}$ and $B_{\nu}$ are taken from the above Example~\ref{DeltaAccQG}. Since $\hat{B}_{\nu}$ is invertible these boundary conditions define quasi--m--accretive Laplace operators. Note that $L_{\nu}$ is a rank one operator and one has that $\Re-\Delta(\hat{A}_{\nu},\hat{B}_{\nu})$ is the operator with $\delta^{\prime}$--couplings and real coupling constants $\Re \gamma_{\nu}$.
\end{example}

An example for which Assumption~\ref{assumption2} is satisfied, but Assumption~\ref{assumptionA} is violated is the following
\begin{example}\label{ex2AccQG}
Let $\Ge=(V,\partial,\Ee)$ be a graph consisting of two external edges $\Ee=\{e_1,e_2\}$ and one vertex $\partial(e_1)=\partial(e_2)$. Consider the boundary conditions defined by 
\begin{eqnarray*}
A_{\tau}=\begin{pmatrix} 1 & -e^{i\tau} \\ 0 & 0\end{pmatrix} &\mbox{and} & B_{\tau}=\begin{pmatrix}  0 & 0 \\ 1 & -e^{-i\tau} \end{pmatrix},
\end{eqnarray*} 
for $\tau \in [0,\pi/2]$. Assumption~\ref{assumption2} is clearly satisfied for all $\tau \in [0,\pi/2]$. Explicit computations give 
\begin{eqnarray*}
\Ran B_{\tau} = \Span \left\{  \begin{pmatrix} 0 \\ 1  \end{pmatrix}\right\},
& & (\Ran B_{\tau})^{\perp} = \Span \left\{  \begin{pmatrix} 1 \\ 0  \end{pmatrix}\right\}, \\ 
 \Ker B_{\tau} = \Span \left\{  \begin{pmatrix} e^{-i\tau} \\ 1  \end{pmatrix}\right\}, & &
 (\Ker B_{\tau})^{\perp} = \Span \left\{  \begin{pmatrix} 1 \\ -e^{-i\tau}  \end{pmatrix}\right\}
\end{eqnarray*}
and therefore 
\begin{align*}
Q_{\tau}A_{\tau}P_{\tau}^{\perp}= \frac{1}{\sqrt{2}}  \begin{pmatrix} 2 & -2 e^{i\tau} \\ 0 & 0   \end{pmatrix} \neq 0.
\end{align*}
The case $\tau=0$ has been discussed in \cite[Example XIX.6.c]{DSIII} already. The function $\psi(k,\cdot)$ defined by
\begin{equation*}
\psi(k,x)=\begin{cases} e^{ikx}, & \mbox{if} \ x\in e_1 , \\  e^{ikx}, & \mbox{if} \ x\in e_2  \end{cases}
\end{equation*}
satisfies the boundary conditions defined by $A_0,B_0$ for any $k$. For $\Im k>0$ one has $$-\Delta(A_0,B_0)\psi(k,\cdot)=k^2\psi(k,\cdot)$$ and hence the operator $-\Delta(A_0,B_0)$ has empty resolvent set. Identifying the graph $\Ge$ with the real line the operator $-\Delta(A_0,B_0)$ corresponds to the operator $$-\sign(x)\tfrac{d}{dx}\sign(x)\tfrac{d}{dx}$$ with its natural domain in $L^2(\R)$, which is self--adjoint in an appropriate Krein space, see for example \cite[Chapter 4, Section 1]{Ich}.
\end{example}

\section{Proofs of the main results}

The proofs of the characterizations consist of two directions. For the ``if-part'' one requires both Assumptions~\ref{assumption2} and \ref{assumptionA}. Denote by $\We_j$, $j\in \Ee \cup \Ie$ the set of all functions $\psi_j\in \He_j$ which are absolutely continuous with square integrable derivative. One sets
$$\We=\bigoplus_{j\in \Ee \cup \Ie} \We_j.$$

Inserting the representation \eqref{eq: Decomposition2} into the quadratic form given in \eqref{eq:qfAccQG} one obtains
\begin{align}\label{qf2AccQG}
\delta_{\Me}[\psi] = \int_{\Ge} \abs{ \psi^{\prime}}^2 - \langle L P^{\perp}\underline{\psi},P^{\perp}\underline{\psi}\rangle_{\Ke}.
\end{align}
Recall that a form $\mathfrak{t}$ is called \textit{sectorial} if there exists a $C\in \R$ such that 
\begin{eqnarray*}
\abs{\Im \mathfrak{t}[u,u]} \leq \Re \mathfrak{t}[u,u] + C \langle u,  u\rangle, & \mbox{for all} & u \in \dom(\mathfrak{t}).
\end{eqnarray*}
From \eqref{qf2AccQG} one deduces
\begin{lemma}\label{lemma2AccQG}
Under the Assumptions~\ref{assumption2} and \ref{assumptionA} the form $\delta_{\Me}$ given in \eqref{eq:qfAccQG} is sectorial. The form domain of its closure is 
\begin{align*}
\dom \, \clo(\delta_{\Me})= \{ \psi\in \We \mid P\underline{\psi}=0\}.
\end{align*}
\end{lemma}
This result gives reason to the definition of the quadratic form $\bar{\delta}_{L,P}$ given by
\begin{align}\label{deltaAccQG}\index{$\bar{\delta}_{L,P}$}
\bar{\delta}_{L,P}[\psi]:= \int_{\Ge} \abs{\psi^{\prime}}^2 - \langle P^{\perp}\underline{\psi},L P^{\perp}  \underline{\psi}\rangle_{\Ke}
\end{align}
with the form domain $\dom(\bar{\delta}_{L,P})=  \{ u\in \We \mid P\underline{u}=0\}$. Obviously $\delta_{\Me}\subset\bar{\delta}_{L,P}$, for $\Me=\Ker(A, \, B)$, where $P$ is the orthogonal projector onto $\Ker B$ and $L$ is computed from $A,B$ by formula \eqref{KuchmentL}.

The proof of Lemma~\ref{lemma2AccQG} uses the following elementary but important trace estimate, which is borrowed from \cite{PKQG1}.   

\begin{lemma}[{Trace estimate, \cite[Lemma 8]{PKQG1}}]\label{lemma1AccQG}
Let $f \in  H^1([0,a)]$. Then 
$$\abs{f(0)}^2\leq \frac{2}{l} \norm{f}^2_{L^2([0,a])}+ l \norm{f^{\prime}}^2_{L^2([0,a])},$$
holds for any $0<l\leq a$, where $H^1([0,a)]\subset L^2([0,a])$ denotes the Sobolev space of first order. 
\end{lemma}
The statement of the lemma remains valid for $f \in  H^1([0,\infty))$ with $0<l<\infty$.

\begin{proof}[Proof of Lemma~\ref{lemma2AccQG}]
The first step is to prove that the quadratic form $\bar{\delta}_{L,P}$ is sectorial with some vertex $C$. The operator $L$ is bounded in the finite dimensional space $\Ke$ and therefore it is even sectorial, as well as $-L$. Hence for sufficiently large $C>0$ there exists a $\gamma>1$ such that the inequality
\begin{align*}
\abs{\Im \langle  \underline{\psi},L\underline{\psi} \rangle_{\Ke} } \leq -\gamma \Re \langle  \underline{\psi},L \underline{\psi} \rangle_{\Ke} + C \langle \underline{\psi},\underline{\psi} \rangle_{\Ke},  
\end{align*}
holds for all $\underline{\psi}\in\Ke$. Adding the positive quantity $0\leq \norm{\psi^{\prime}}_{\He}^2$ on the right hand side gives
\begin{align*}
\abs{\Im \langle \underline{\psi},L\underline{\psi} \rangle_{\Ke}} \leq -\gamma \Re \langle \underline{\psi},L\underline{\psi} \rangle_{\Ke} + C \langle \underline{\psi},\underline{\psi} \rangle_{\Ke} +\norm{\psi^{\prime}}_{\He}^2, 
\end{align*}
for all $\psi\in \We$. The trace estimate in Lemma~\ref{lemma1AccQG} gives for sufficient small $l>0$ with $l\leq \min_{i\in \Ie}a_i$,   
\begin{align*}
C \langle \underline{\psi},\underline{\psi} \rangle_{\Ke} + \norm{\psi^{\prime}}_{\He}^2  \leq  (1+2 Cl)\cdot \norm{\psi^{\prime}}_{\He}^2 + 2 \frac{2C}{l} \norm{\psi}_{\He}^2,
\end{align*}
where the trace estimate has been applied to all endpoint of the edges. This leads to the inequality
\begin{align*}
\abs{\Im \langle  \underline{\psi},L \underline{\psi} \rangle_{\Ke}} \leq  -\gamma \Re \langle \underline{\psi},L \underline{\psi} \rangle_{\Ke} + (1+2Cl) \cdot \norm{\psi^{\prime}}_{\He}^2 + \frac{4C}{l} \norm{\psi}_{\He}^2, && \psi\in \We. 
\end{align*}
Choosing $l>0$ so small that $(1+Cl)\leq \gamma$ for a fixed $C>0$ one can estimate the right hand side of the above inequality and one arrives at
\begin{align*}
\abs{\Im \langle  \underline{\psi},L \underline{\psi} \rangle_{\Ke}} \leq  \gamma \Re \bar{\delta}_{P,L}[\psi] + \frac{4C}{l} \norm{\psi}_{\He}^2, && \psi \in \We.
\end{align*}
Therefore the form $\bar{\delta}_{P,L}$ is sectorial even in the larger domain $\We$. From the domain inclusions  
$$\Dom(-\Delta(A,B))\subset \dom(\bar{\delta}_{P,L}) \subset \We$$
it follows that $\delta_{\Me}$ is sectorial, too. From Proposition~\ref{prop1AccQG} it follows that $-\Delta(A,B)$, where $\Me=\Me(A,B)$, is quasi--m--accretive and therefore even m--sectorial. It remains to determine the closure of the form associated with $-\Delta(A,B)$. 

As $\delta_{\Me}$ is sectorial there exists a constant $C>0$ such that 
$$\norm{\psi}_{\delta_{\Me}}:=\left(\Re \delta_{\Me}[\psi]+C \norm{\psi}_{\He}^2\right)^{1/2}$$ 
defines a norm in $\Dom -\Delta(A,B)$. This norm $\norm{\cdot}_{\Delta_{\Me}}$ is equivalent to the Sobolev norm $\norm{\cdot}_{\We}$, which is given by the formula
$$\norm{\psi}_{\We}^2= \norm{\psi}_{\He}^2+\norm{\psi^{\prime}}_{\He}^2.$$ 
The proof is analogue to the proof of the sectoriality of $\delta_{\Me}$. By the same reasoning it follows that there exists a $C>0$ such that for all $\psi \in \We$
\begin{align*}
0 \leq \Re \langle \underline{\psi},-L\underline{\psi} \rangle_{\Ke}  +\tfrac{1}{2}\norm{\psi^{\prime}}_{\He}+\tfrac{1}{2}\norm{\psi}_{\He} +(C-1)\norm{\psi}_{\He}
\end{align*}
holds. Adding the positive term $\tfrac{1}{2}\norm{\psi^{\prime}}_{\He}+\tfrac{1}{2}\norm{\psi_{\He}}$ on both sides one gets
\begin{align*}
\tfrac{1}{2}\norm{\psi}_{\We}^2 \leq  \norm{\psi}^2_{\delta_{\Me}}.
\end{align*} 
On the other hand applying Lemma~\ref{lemma1AccQG} one has for a constant $C^{\prime}>0$
\begin{align*}
\Re \langle \underline{\psi},-L\underline{\psi} \rangle_{\Ke}  +\norm{\psi^{\prime}}_{\He}+C\norm{\psi}_{\He}\leq  C^{\prime}\norm{\psi}_{\We}.
\end{align*}
Hence $\norm{\cdot}_{\delta_{\Me}}$ and $\norm{\cdot}_{\We}$ are equivalent in $\We_P=\{ \psi \in  \We \mid P \underline{\psi}=0\}$. For the closure of $\Dom(\Delta(A,B))$ with respect to the norms $\norm{\cdot}_{\delta_{\Me}}$ and $\norm{\cdot}_{\We}$ one obtains
\begin{align*} 
\overline{\Dom(-\Delta(A,B))}^{\norm{\cdot}_{\delta_{\Me}}}=\overline{\Dom(-\Delta(A,B))}^{\norm{\cdot}_{\We}}.
\end{align*}
Note that $\Dom(-\Delta(A,B))$ is a dense subset of $\dom(\bar{\delta}_{P,L})$ and $\dom(\bar{\delta}_{P,L})$ is a closed subspace of $\We$. Therefore one arrives at the conclusion that the closure of the form $\delta_{\Me}$ is $\bar{\delta}_{L,P}$. 
\end{proof}

To prove the ``only-if'' part it is sufficient to show that assuming Assumption~\ref{assumption2} and $QAP^{\perp}\neq 0$ gives that the numerical range of operator $-\Delta(A,B)$ contains the whole real line and therefore it cannot be quasi-accretive. 


\begin{lemma}\label{lemma4AccQG}
Let one of the Assumptions~\ref{assumption2} and \ref{assumptionA} be violated. Then $-\Delta(A,B)$ fails to be quasi--m--accretive. 
\end{lemma}

\begin{proof}

Let Assumption~\ref{assumption2} be violated. Then by Proposition~\ref{prop1AccQG} the operator $-\Delta(A,B)$ fails to be quasi--m--accretive. Therefore, suppose that Assumption~\ref{assumption2} holds and Assumption~\ref{assumptionA} is violated.

Consider for simplicity the parametrization~\eqref{eq: Decomposition} instead of the parametrization~\eqref{eq: Decomposition1}. Inserting the boundary condition \eqref{eq: Decomposition} into the quadratic form \eqref{eq:qfAccQG} yields
\begin{align*}
\delta_{\Me}[\psi] = \int_{\Ge} \abs{ \psi^{\prime}}^2 - 
\langle P^{\perp} \underline{\psi}, P^{\perp}A^{\prime}  \underline{\psi}\rangle_{\Ke}
+\langle P\underline{\psi},P\underline{\psi}^{\prime}\rangle_{\Ke},
\end{align*}
where $\Me(A,B)=\Me(A^{\prime},B^{\prime})$ and $P$ is the orthogonal projector onto $\Ker B^{\prime}$. Since $QA$ is surjective by Assumption~\ref{assumption2} also $PA^{\prime}$ is surjective, and one has that 
$$\dim \Ker PA^{\prime} = \dim \Ran P^{\perp}.$$ 
Therefore $QAP^{\perp}\neq 0$ implies $PA^{\prime}P^{\perp}\neq 0$ which delivers $\Ker PA^{\prime} \neq \Ran P^{\perp}$. This in turn implies that there is a vector $\alpha$ such that 
\begin{eqnarray}\label{alphaAccQG}
PA^{\prime} \,\alpha=0, & \mbox{but} & P \,\alpha\neq 0.
\end{eqnarray}
Using this vector one constructs explicitly a sequence $u_n\in \Dom (-\Delta(A^{\prime},B^{\prime}))$ such that $\Re \langle u_n, -\Delta(A^{\prime}, B^{\prime})u_n \rangle \rightarrow -\infty,$ for $n\to \infty$. For simplicity suppose that $\Ge$ is a finite star graph, that is $\Ie=\emptyset$ and $\abs{\Ee}=m$. 

First one defines an auxiliary matrix-valued function on the half-line. Let $0<a<b<c$ be positive numbers, $H\in \mbox{End}(\Ke)$ an arbitrary matrix and $p(H;x)$ and $q(H;x)$ functions in $x$ and $H$. One defines $$\Phi_H{[p,q;a,b,c]}\colon [0,\infty) \rightarrow \mbox{End}(\Ke),$$
by
\begin{align*}
\Phi_H{[p,q;a,b,c]}(x)= \begin{cases} e^{Hx}, & x\in [0,a], \\  
                            p(H;x), & x\in (a,b), \\
                            q(H;x), &      x\in [b,c), \\
                              0, & x\in[c,\infty).
               \end{cases}
\end{align*}
Consider for $n\geq 1$ the sequence of matrices 
\begin{align*}
H_n=-\begin{pmatrix} P^{\perp} A^{\prime} P^{\perp} & P^{\perp} A^{\prime} P \\ 0 & n P    \end{pmatrix}
\end{align*}
and note that
\begin{eqnarray*}
2 \Re H_n=\begin{pmatrix} 2 \Re P^{\perp} A^{\prime} P^{\perp} & P^{\perp} A^{\prime} P \\ (P^{\perp} A^{\prime} P)^* & 2n P    \end{pmatrix} 
\end{eqnarray*}
and that
\begin{eqnarray*}
2 \Im H_n=\begin{pmatrix} 2 \Im P^{\perp} A^{\prime} P^{\perp} & P^{\perp} A^{\prime} P \\ -(P^{\perp} A^{\prime} P)^* & 0    \end{pmatrix}.
\end{eqnarray*}
In order to analyse these bounded block operator matrices the concept of the quadratic numerical range is helpful, for further information on this topic see the book \cite{Tretter} and the references therein. As the numerical ranges of the diagonal blocks are contained in the numerical range of the whole block operator matrix one obtains that $\norm{\Re H_n} \sim n $ for large $n$. Therefore also 
\begin{eqnarray*}
\norm{H_n}= \left(\norm{\Re H_n}^2 + \norm{\Im H_n}^2  \right)^{1/2}\sim n & & \mbox{for $n$ large}.
\end{eqnarray*}

Define now for sequences $(a_n),(b_n),(c_n)$ with $0<a_n<b_n<c_n$ for $n\geq 1$ the polynomial $p_n(H_n;\cdot)$ in $x$ by
\begin{equation*}
p_n(H_n;x)= \beta_n \frac{(x-b_n)^{n+1}}{n+1}+ \gamma_n
\end{equation*}
with coefficients
\begin{eqnarray*}
\beta_n= \frac{H_n e^{H_na_n}}{(a_n-b_n)^n} &\mbox{and} & \gamma_n= \left[ \mathds{1} - \frac{(a_n-b_n)}{n+1}H_n \right] e^{H_n a_n}.
\end{eqnarray*}
This assures that 
\begin{eqnarray*}
p_n(H_n;a_n)=e^{Ha_n}, & & \tfrac{d}{dx} p_n(H_n;a_n)= H_n e^{H_na_n}, \\
p_n(H_n;b_n)=\gamma_n & \mbox{and} & \tfrac{d}{dx}  p_n(H_n;b_n)= 0.
\end{eqnarray*}
Furthermore choose $q_n$ to be a polynomial in $x$ such that 
\begin{eqnarray*}
q_n(H_n;b_n)= \gamma_n, & & \tfrac{d}{dx} q_n(H_n;b_n)=0, \\
q_n(H_n;c_n)= 0 &\mbox{and} & \tfrac{d}{dx}q_n(H_n;c_n)=0 
\end{eqnarray*}
hold for all $n\geq 1$. This gives that the function $\Phi_{H_n}{[p_n,q_n;a_n,b_n,c_n]}$ is a function in the Sobolev space $H^2([0,\infty), \mbox{End}(\Ke))$ for all $n\in \N$, and by construction it is even compactly supported.

Now with the vector $\alpha$ chosen above, compare \eqref{alphaAccQG}, one sets
\begin{eqnarray*}
\{u_n\}_e(x):=  \big\{\Phi_{H_n}[p_n,q_n;a_n,b_n,c_n](x) \alpha\big\}_{e}, & \mbox{for } e\in \Ee. 
\end{eqnarray*}
This defines functions $u_n\colon \Ge \rightarrow \C$ for $n\in\N$. By construction one has $u_n\in \De$ . One proves now that $u_n\in \Dom(-\Delta(A^{\prime},B^{\prime}))$. Indeed, 
\begin{eqnarray*}
\Phi_{H_n}[p_n,q_n;a_n,b_n,c_n](0) \alpha= \alpha & \mbox{and} & \frac{d}{dx} \Phi_{H_n}[p_n,q_n;a_n,b_n,c_n](\cdot) \alpha \Big|_{x=0}= H_n\alpha.
\end{eqnarray*}
Therefore
\begin{eqnarray*}
\underline{u_n}= \begin{pmatrix} P^{\perp}\alpha \\ P \alpha  \end{pmatrix}
& \mbox{and} & \underline{u_n}^{\prime}= -\begin{pmatrix} P^{\perp} A^{\prime} P^{\perp} & P^{\perp} A^{\prime} P \\ 0 & n P    \end{pmatrix} \begin{pmatrix} P^{\perp}\alpha \\ P \alpha  \end{pmatrix}
\end{eqnarray*}
for all $n\in\N$. From $\underline{u_n}=\alpha$, $\begin{pmatrix} PA^{\prime}P^{\perp} & PA^{\prime}P \end{pmatrix} \alpha=0$ and \eqref{alphaAccQG} it follows that
\begin{align*}
A^{\prime} \underline{u_n} + B^{\prime} \underline{u_n}^{\prime}= 
   \begin{pmatrix} P^{\perp} A^{\prime} P^{\perp} & P^{\perp} A^{\prime} P \\ P A^{\prime} P^{\perp} & PA^{\prime} P  \end{pmatrix} \begin{pmatrix} P^{\perp}\underline{u_n} \\ P \underline{u_n}  \end{pmatrix} +
   \begin{pmatrix}  P^{\perp} & 0 \\ 0& 0 \end{pmatrix} \begin{pmatrix} P^{\perp}\underline{u_n}^{\prime} \\ P \underline{u_n}^{\prime}  \end{pmatrix} 
     =0.
\end{align*}
This implies that $u_n\in \Dom(-\Delta(A^{\prime},B^{\prime}))$ for all $n\in \N$. Inserting $u_n$ into the quadratic form \eqref{eq:qfAccQG} gives
\begin{align*}
\langle u_n, -\Delta(A^{\prime},B^{\prime}) u_n \rangle =& \int_{\Ge} \abs{u_n^{\prime}}^2 + \langle \underline{u_n}, \underline{u_n}^{\prime} \rangle_{\Ke} \\
 =& \int_{\Ge} \abs{u_n^{\prime}}^2 -   \langle  P^{\perp}\alpha, P^{\perp} A^{\prime} \alpha \rangle_{\Ke}  - n \langle P \alpha, P \alpha \rangle_{\Ke}.
\end{align*}
The term $- \langle  P^{\perp}\alpha, P^{\perp} A^{\prime} \alpha \rangle_{\Ke}$ is bounded. Chose now 
\begin{eqnarray*}
a_n=\frac{e^{-2 \norm{H_n}}}{\norm{H_n}^2}, &  \displaystyle{ b_n=2 \frac{e^{-2 \norm{H_n}}}{\norm{H_n}^2} }& \mbox{and } c_n=\mbox{constant},
\end{eqnarray*}
for $n$ sufficiently large. With this one obtains 
\begin{align*}
\int_{a_n}^{b_n}  \left\vert \tfrac{d}{dx}(p_n)(H_n;x)\alpha \right\vert^2 dx &\leq m \frac{\norm{H_n}^2 }{\abs{b_n-a_n}^{2n}} \norm{e^{H_n}}^2 \norm{\alpha}^2 \int_{a_n}^{b_n}  \abs{x-b_n}^{2n}   dx \\
                              &\leq   m \frac{\norm{H_n}^2 e^{2\norm{H_n}} \norm{\alpha}^2   }{(2n+1)} \abs{a_n-b_n} \rightarrow 0, \quad \mbox{for } n \to \infty,
\end{align*}
where $m=\abs{\Ee}$. Furthermore one has
\begin{align*}
\int_0^{a_n}  \abs{ H_n e^{H_nx}\alpha }^2 dx &\leq m \norm{\alpha}^2 \norm{ H_n}^2 \int_0^{a_n}  e^{2\norm{H_n}x} dx  \\
&= \frac{m}{2} \norm{\alpha}^2 \norm{ H_n} \left( e^{2\norm{H_n}a_n}-1\right) \rightarrow 0, \quad \mbox{for } n \to \infty.
\end{align*}
Since $\norm{\gamma_n} \rightarrow 1$ for $n\to \infty$ and $\tfrac{d}{dx}  p_n(H;b_n)= 0$, there exist constants $C,c>0$ such that
\begin{align}\label{Cc}
c \leq \int_{b_n}^c  \abs{ \tfrac{d}{dx}(q_n(H_n;x))\alpha }^2 dx \leq C
\end{align}
holds for all $n\in \N$. Together this gives that $\int_{\Ge} \abs{u_n^{\prime}}^2$ is uniformly bounded  whereas \\ $ \Re \langle \underline{u_n}, \underline{u_n}^{\prime} \rangle \to -\infty$ and therefore
\begin{eqnarray*}
\Re \langle u_n, -\Delta(A^{\prime},B^{\prime}) u_n \rangle   \rightarrow -\infty , & \mbox{for} \ n\to\infty.
\end{eqnarray*}
Now one estimates the $L^2$--norm of $u_n$. Since $u_n \in H^2([0,c_n], \mbox{End}(\Ke))$, $u_n(c_n)=0$ and $c_n=\mbox{constant}$ for $n$ sufficiently large it follows that a Poincar\'{e} inequality can be applied to $u_n$ and hence there is a  constant $C>0$ which is uniform in $n$ such that 
\begin{eqnarray*}
\int_{\Ge} \abs{u_n}^2 \leq C  \int_{\Ge} \abs{u_n^{\prime}}^2 & \mbox{for all } n \in \N.
\end{eqnarray*}
At the same time one has that
\begin{eqnarray*}
0<\int_{b_n}^{c_n} \abs{q_n(H_n;x)\alpha}^2  < \int_{\Ge} \abs{u_n}^2  & \mbox{for all } n \in \N.
\end{eqnarray*}
Hence $\norm{u_n}$ is uniformly bounded from below and from above. Consequently the operator $-\Delta(A,B)=-\Delta(A^{\prime},B^{\prime})$ is not quasi--accretive. This proves Lemma~\ref{lemma4AccQG} for the case of star graphs. The construction of $\Phi_{H_n}[p_n,q_n;a_n,b_n,c_n](\cdot)$ has been done only for simplicity on the half line. Actually only the locality of the boundary conditions is needed. Restricting the functions $u_n$ to small neighbourhoods of the vertices carries the proof over to arbitrary finite metric graphs. Locality of the boundary conditions can be achieved always by collapsing all vertices into one single vertex. This method of ``localisation'' has been used frequently in the literatur, for example recently in \cite{Davies}. 
\end{proof}

\begin{proof}[Proof of Theorem~\ref{th1AccQG}]
If both Assumptions~\ref{assumption2} and \ref{assumptionA} are satisfied then by \\ Lemma~\ref{lemma2AccQG} the operator $-\Delta(A,B)$ is quasi--m--accretive and even m--sectorial. If one of the assumptions is violated it follows from Lemma~\ref{lemma4AccQG} that $-\Delta(A,B)$ fails to be quasi--m--accretive.  
\end{proof}

\begin{lemma}\label{lemma3AccQG}
Under the Assumptions~\ref{assumption2} and \ref{assumptionA} the operator $-\Delta(A,B)$ is m--sectorial and the symmetric form $\Re \left(\clo \, \delta_{\Me}\right)$ with $\Me=\Me(A,B)$ corresponds to the self-adjoint Laplace operator $-\Delta(A^{\prime},B^{\prime})$ with $A^{\prime}=P+ \Re L$ and $B^{\prime}= P^{\perp}$, where $P$ is the orthogonal projector onto $\Ker B$, $P^{\perp}=\mathds{1}-P$ and $L$ is computed from $A,B$ by formula~\eqref{KuchmentL}. 
\end{lemma}
\begin{proof}[Proof]
The form $\Re \clo \left( \delta_{\Me}\right)= \Re \bar{\delta}_{P,L}$ is symmetric and closed. Hence according to the first representation theorem, compare for example \cite[Theorem VI.2.1]{Kato}, it corresponds to a self-adjoint operator. All forms of this type and their corresponding self-adjoint operators has been described in \cite[Theorems 6 and 9]{PKQG1}. Since $\Re \bar{\delta}_{P,L}= \bar{\delta}_{P, \Re L}$ the operator defined by $\Re \clo \left( -\Delta_{\Me}\right)$ agrees with the Laplace operator $-\Delta(P + \Re L, P^{\perp})$. 
\end{proof}

\begin{proof}[Proof of Theorem~\ref{th2AccQG}]
The first statement follows immediately from Lemma~\ref{lemma3AccQG}. By \cite[Theorem 1]{AL2009} a self-adjoint Laplace operator $-\Delta(A_{sa},B_{sa})$ is non-negative if and only if
\begin{eqnarray*}
A_{sa}B_{sa}^* + B_{sa} M_0(\au) B_{sa}^*  \leq 0,
\end{eqnarray*} 
where $M_0(\au)$ is the matrix given in Theorem~\ref{th2AccQG}. Therefore the operator $-\Delta(A,B)$ is m--accretive if and only if the Assumptions~\ref{assumption2} and \ref{assumptionA} are satisfied and the form $\Re \bar{\delta}_{(P,L)}$ is non-negative. According to Lemma~\ref{lemma3AccQG} this is the case if and only if $-\Delta(A^{\prime},B^{\prime})\geq 0$ with $A^{\prime}=P+ \Re L$ and $B^{\prime}= P^{\perp}$. Using the cited result this is equivalent to the condition 
\begin{equation*}
\Re AB^* + B M_0(\au) B^*  \leq 0.
\end{equation*} 
It remains to show that $\Re AB^* + B M_0(\au) B^*  \leq 0$ implies that Assumption~\ref{assumptionA} holds. Let Assumption~\ref{assumption2} be satisfied. Then one can assume without loss of generality $B=P^{\perp}$. The decomposition of $A$ with respect to $P$ and $P^{\perp}$ gives
\begin{align*}
AB^*=\begin{pmatrix}
P^{\perp}AP^{\perp} & P^{\perp}A P \\
PAP^{\perp} & PAP
\end{pmatrix}
\begin{pmatrix}
P^{\perp} & 0 \\
0 & 0
\end{pmatrix}= \begin{pmatrix}
P^{\perp}AP^{\perp} & 0 \\
PAP^{\perp} & 0
\end{pmatrix}
\end{align*}
and, hence,
\begin{eqnarray*}
\Re(AB^*)=\frac{1}{2}\begin{pmatrix}
2 \Re(P^{\perp}AP^{\perp}) & (PAP^{\perp})^* \\
PAP^{\perp} & 0
\end{pmatrix}
\end{eqnarray*}
and
\begin{eqnarray*}
BM_0(\au)B^*=\begin{pmatrix}
P^{\perp}M_0(\au)P^{\perp} & 0 \\
0 & 0
\end{pmatrix}.
\end{eqnarray*}
The numerical range of such block-matrix operators is discussed in the following elementary lemma. It covers a particular case of the problem of the positive completion of diagonal block-operator matrices, see \cite{Hou} and the references therein, where the general problem is discussed.
\begin{lemma}\label{QWB}
Let $M$ be a bounded self-adjoint block-matrix operator of the form
\begin{eqnarray*}
M=\begin{pmatrix}
A & B^* \\
B & 0
\end{pmatrix}, & \mbox{where} & A=A^*.
\end{eqnarray*}
Then $M\leq 0$ ( $M\geq 0$) if and only if $A\leq 0$ ($A\geq 0$)  and $B\equiv 0$.
\end{lemma}
The proof of this lemma makes use of the concept of the quadratic numerical range. For further references on this topic the author highly recommends the book \cite{Tretter}. 

Applying Lemma~\ref{QWB} to the boundary conditions defined by $(A,\,B)$ gives that $\Re(AB^*)+BM_0(\au)B^*\leq 0$ if and only if $PAP^{\perp} \equiv 0$, which is nothing but Assumption~\ref{assumptionA}.
\end{proof}

\begin{proof}[Proof of Lemma \ref{QWB}]
The numerical range of $A$ is included in the numerical range of $M$ therefore a necessary condition for $M$ to be negative definite is that $A$ is negative definite. Assume now that $B\neq0$. Then there exists $u$ and $v$ such that $b=\langle Bu,v\rangle\neq 0$. Consider the $2\times 2$--matrix 
\begin{align*}
M(u,v)= \begin{bmatrix} \langle A u,u\rangle & \langle Bu,v\rangle \\ \langle B^*v,u\rangle & 0 \end{bmatrix} = \begin{bmatrix} a & b \\ b^* & 0 \end{bmatrix}
\end{align*}
with $a=\langle A u,u\rangle$, which has the two eigenvalues 
\begin{eqnarray*}
\lambda_+= \frac{1}{2} a + \frac{1}{2} \sqrt{a^2 + 4\abs{b}^2} &\mbox{and} & \lambda_-= \frac{1}{2} a - \frac{1}{2} \sqrt{a^2 + 4\abs{b}^2}.
\end{eqnarray*}
The number $\lambda_{-}$ is negative whereas $\lambda_{+}$ is positive and hence the numerical range of $M$ takes positive as well as negative values. The endpoints of the numerical range are in the spectrum of the self--adjoint operator $M$, which means that $M$ is indefinite. Assuming the other way around that $A\leq 0$ and $B\equiv 0$ the statement follows. For $M\geq 0$ the proof is analogue.
\end{proof}

\textsc{Amru Hussein, FB 08 - Institut f\"{u}r Mathematik,} \\ \textsc{Johannes Gutenberg--Universit\"{a}t Mainz,}\\ \textsc{Staudinger Weg 9, 55099 Mainz, Germany}
\\
\textit{Email adress:} {\fontfamily{pcr}\selectfont hussein@mathematik.uni-mainz.de}

\end{document}